\pdfoutput=1
\RequirePackage{ifpdf}
\ifpdf 
\documentclass[pdftex]{sigma}
\else
\documentclass{sigma}
\fi

\def\bm{\mathbf}

\numberwithin{equation}{section}

\newtheorem{Theorem}{Theorem}[section]
\newtheorem{Corollary}[Theorem]{Corollary}

\newtheorem{Proposition}[Theorem]{Proposition}

\begin{document}

\newcommand{\arXivNumber}{1807.08184}

\renewcommand{\PaperNumber}{004}

\FirstPageHeading

\ShortArticleName{Relations between Schoenberg Coefficients on Real and Complex Spheres}

\ArticleName{Relations between Schoenberg Coefficients\\ on Real and Complex Spheres of Different Dimensions}

\Author{Pier Giovanni BISSIRI~$^\dag$, Valdir A.~MENEGATTO~$^\ddag$ and Emilio PORCU~$^{\dag\S}$}

\AuthorNameForHeading{P.G.~Bissiri, V.A.~Menegatto and E.~Porcu}

\Address{$^\dag$~School of Mathematics $\&$ Statistics, Newcastle University,\\
\hphantom{$^\dag$}~Newcastle Upon Tyne, NE1 7RU, UK}
\EmailD{\href{pier-giovanni.bissiri@newcastle.ac.uk}{pier-giovanni.bissiri@newcastle.ac.uk}, \href{mailto:emilio.porcu@newcastle.ac.uk}{emilio.porcu@newcastle.ac.uk}}

\Address{$^\ddag$~Instituto de Ci\^encias Matem\'aticas e de Computa\c{c}\~ao, Universidade de S\~ao Paulo,\\
\hphantom{$^\ddag$}~Caixa Postal 668, 13560-970, S\~ao Carlos - SP, Brazil}
\EmailD{\href{menegatt@icmc.usp.br}{menegatt@icmc.usp.br}}

\Address{$^\S$~Department of Mathematics, University of Atacama, Copiap\'o, Chile}

\ArticleDates{Received July 25, 2018, in final form January 18, 2019; Published online January 23, 2019}

\Abstract{Positive definite functions on spheres have received an increasing interest in many branches of mathematics and statistics. In particular, the Schoenberg sequences in the spectral representation of positive definite functions have been studied by several mathe\-maticians in the last years. This paper provides a set of relations between Schoenberg sequences defined over real as well as complex spheres of different dimensions. We illustrate our findings describing an application to strict positive definiteness.}


\Classification{33C45; 42A16; 42A82; 42C10}

\section{Introduction}\label{s-introd}

Positive definite functions on real and complex spheres have a long history, that starts with the seminal paper by Schoenberg \cite{schoenberg}. Positive definiteness is crucial to many branches of mathematical analysis
\cite{Barbosa,beatson,chen,gangolli,Guella3,guella1,Guella2,hannan,Menegatto-strict,Menegatto-strict2,menegatto-peron,schoenberg} and statistics \cite{baldi-marinucci,berg-porcu,erosthe2,erosthe, bogaert,papanicolau,clarke,gerber,Hansen,hitz,huang,istas,lang-schwab,leonenko,Maliarenko,moller,PBG}. Recent reviews on positive definite functions on either spheres or product spaces involving spheres can be found in \cite{gneiting2} and in \cite{PAF18} as well.

Fourier analysis on spheres is related to the so called Schoenberg sequences (also called sequences of Schoenberg coefficients in~\cite{daley-porcu}) that are related to the dimension where any positive definite function on real or complex spheres is defined. There has been a recent interest on Schoenberg sequences, especially after the list of research problems in~\cite{gneiting2} and in~\cite{PAF18}. Recursive relations between Schoenberg coefficients on $d$-dimensional spheres have been first proposed by~\cite{gneiting2}. Fiedler~\cite{fiedler} has then solved an open problem in~\cite{gneiting2}, related to other types of recursions involving Schoenberg coefficients. Ziegel \cite{ziegel} has used Schoenberg sequences to find the convolution roots of positive definite functions on spheres and illustrated the differentiability properties of positive definite functions on spheres through their Schoenberg sequences in~\cite{ziegel2}. Recently, Arafat et al.~\cite{arafat2} have solved Gneiting's research problem number 3 making extensive use of Schoenberg sequences. Projections from Hilbert into finite-dimensional spheres have been considered in~\cite{moller}. Finally, Schoenberg sequences have been shown to be central to the study of geometric properties of Gaussian fields on spheres~\cite{lang-schwab} or spheres cross time~\cite{clarke}.

Literature on complex spheres has been sparse. After the tour de force in~\cite{Menegatto14} there has been a recent interest on complex spheres as reported from~\cite{peron2} and in~\cite{massa}.

This paper is about Schoenberg sequences on spheres of $\mathbb{R}^d$ and $\mathbb{C}^{q}$, respectively. Specifically, we show recursive relations that have been lacking from the previously mentioned literature. Section~\ref{2} deals with real-valued $d$-dimensional spheres. Section~\ref{3} is instead related to complex spheres. Some implications in terms of strict positive definiteness are provided in Section~\ref{4}. The paper ends with a discussion.

\section{Schoenberg sequences on real spheres} \label{2}

\subsection{Background and notation}

For a positive integer $d$, let $\mathbb{S}^d=\big\{\bm{x}\in\mathbb{R}^{d+1},\, \|\bm{x}\|=1\big\}$ denote the $d$-dimensional unit sphere embedded in $\mathbb{R}^{d+1}$, with $\|\cdot\|$ being the Euclidean norm.
We define the great-circle distance $\theta\colon \mathbb{S}^d \times \mathbb{S}^d \to [0,\pi] $ as the continuous mapping defined through
\begin{gather*}
\theta(\bm{x}_1,\bm{x}_2) = \arccos\big(\bm{x}_1^\top \bm{x}_2\big) \in[0,\pi],
\end{gather*}
for $\bm{x}_1,\bm{x}_2\in\mathbb{S}^d$, where $\top$ is the transpose operator. A mapping $C\colon \mathbb{S}^d\times \mathbb{S}^d \to \mathbb{R}$ that satisfies
\begin{gather*}
\sum_{i,j=1}^n c_ic_j C(\bm{x}_i,\bm{x}_j) \geq 0
\end{gather*}
for all $n\geq 1$, distinct points $\bm{x}_1, \bm{x}_2, \ldots, \bm{x}_n$ on $\mathbb{S}^d$ and real numbers $c_1, \ldots, c_n$, is called {\em positive definite}. Further, if the inequality is strict, unless the vector $(c_1,\dots,c_n)^{\top}$ is the zero vector, then it is called {\em strictly positive definite} (see~\cite{bingham} and references therein). If, in addition,
\begin{gather}\label{isotropy}
C(\bm{x}_1,\bm{x}_2) = \psi(\theta(\bm{x}_1,\bm{x}_2)), \qquad \bm{x}_i \in \mathbb{S}^d, \qquad i=1,2,
\end{gather}
for some mapping $\psi\colon [0,\pi]\rightarrow \mathbb{R}$, then $C$ is called a {\em geodesically isotropic covariance} by~\cite{PAF18}. With no loss of generality, we assume through the paper that the function $\psi$ is continuous along with the normalization $\psi(0)=1$.

Gneiting \cite{gneiting2} calls $\Psi_d$ the class of continuous functions $\psi\colon [0,\pi] \to \mathbb{R}$ with $\psi(0)=1$ such that the function $C$ in equation~(\ref{isotropy}) is positive definite. The inclusions $\Psi_{d} \supset \Psi_{d+1}$, $d\geq 1$, are known to be strict. Following~\cite{schoenberg}, for every continuous function $\psi\colon [0,\pi]\to \mathbb{R}$ with $\psi(0)=1$, and every integer $d\geq 2$, define
\begin{gather} \label{eq: d-schoenberg}
b_{n,d} = \kappa(n,d) \int_{0}^{\pi} \psi(\theta) C_{n}^{(d-1)/2}(\cos \theta) \left (\sin \theta \right )^{d-1} {\rm d} \theta,
\end{gather}
where, for any $\lambda>0$, $C_{n}^\lambda$ denotes the $n$-th Gegenbauer polynomial of order $\lambda$~\cite{abramowitz1964handbook}, and
\begin{gather}\label{eq: kappa}
\kappa(n,d)=
\frac{(2n+d-1)(\Gamma((d-1)/2)^2}{2^{3-d}\pi\Gamma(d-1)}. \end{gather}
Moreover, we define
\begin{gather} \label{eq: d-schoenberg1}
b_{0,1}=\frac{1}{\pi}\int_0^\pi \psi(\theta) {\rm d}\theta, \qquad b_{n,1} = \frac{2}{\pi}\int_0^\pi \cos(n\theta) \psi(\theta) {\rm d}\theta, \qquad n\geq 1.
\end{gather}
Note that in the cases $d=1$ (the circle) and $d=2$ (the unit sphere of~$\mathbb{R}^3$), Gegenbauer polynomials simplify to Chebyshev and Legendre polynomials \cite{abramowitz1964handbook}, respectively.

The coefficient sequences $\{b_{n,d}\}_{n=0}^\infty$ play a crucial role in the spectral representations for positive definite functions on spheres, which are the equivalent of Bochner and Schoenberg's theorems in Euclidean spaces (see~\cite{{daley-porcu}} with the references therein) and are provided by~\cite{schoenberg}, who shows that a mapping $\psi\colon [0,\pi]\rightarrow \mathbb{R}$ belongs to the class $\Psi_d$ if and only if
it can be uniquely written as
\begin{gather} \label{eq: spectral_rep}
\psi(\theta) = \sum_{n=0}^{\infty} b_{n,d} c_{n}^{(d-1)/2}(\cos \theta), \qquad \theta \in [0,\pi],
\end{gather}
where $c_{n}^\lambda$ denotes the normalized $\lambda$-Gegenbauer polynomial of degree $n$,
namely,
\begin{gather*}
c_{n}^\lambda(u)=\frac{C_{n}^\lambda(u)}{C_{n}^\lambda(1)},\qquad u\in[-1,1],
\end{gather*}
and $\{b_{n,d}\}_{n=0}^\infty$ is a probability mass sequence. The series \eqref{eq: spectral_rep} is known to be uniformly convergent. We follow \cite{daley-porcu} when calling the sequence $\{ b_{n,d}\}_{n=0}^{\infty}$ in \eqref{eq: spectral_rep} the {\em $d$-Schoenberg sequence of coefficients}, to emphasize the dependence on the index $d$ in the class
$\Psi_d$. Accordingly, we say that $(\psi, \{ b_{n,d}\}) $ is a uniquely determined {\em $d$-Schoenberg pair} if $\psi$ belongs to the class $\Psi_d$ and admits the expansion (\ref{eq: spectral_rep}) with $d$-Schoenberg sequence $\{b_{n,d} \}_{n=0}^{\infty}$.

The following recursive relations among the coefficients $b_{n,d}$ and $b_{n,d+2}$ attached to a $d$-Schoenberg pair $(\psi, \{ b_{n,d+2} \})$ (see \cite[Corollary~1]{gneiting2})
\begin{gather}
b_{0,3}=b_{0,1}-\frac{1}{2}b_{2,1}, \label{eq: rel_gegenb1} \\
b_{n,3}=\frac{1}{2}(n+1)(b_{n,1}-b_{n+2,1}), \qquad n\geq 1, \label{eq: rel_gegenb2}\\
b_{n,d+2}=\frac{(n+d-1)(n+d)}{d(2n+d-1)} b_{n,d}-\frac{(n+1)(n+2)}{d(2n+d+3)}b_{n+2,d}, \qquad d\geq 2,\qquad n\geq 0, \label{eq: rel_gegenb3}
\end{gather}
have actually opened for challenging questions. Fiedler~\cite{fiedler} has shown relationships between sequences $\{b_{n,2d+1}\}_{n=0}^\infty$ and $\{b_{n,1}\}_{n=0}^\infty$, on the one hand, and sequences $\{b_{n,2d}\}_{n=0}^\infty$ and $\{b_{n,2}\}_{n=0}^\infty$, on the other. Proposition~1 in~\cite{arafat2} encompasses Fiedler's result and provides a relation between the sequences $\{b_{n,d}\}_{n=0}^\infty$, $d >1$, and $\{b_{n,1} \}_{n=0}^\infty$. A projection operator relating Schoenberg sequences on Hilbert spheres with $d$-Schoenberg sequences has been proposed by \cite{moller}. Yet, there are some relations that have not been discovered and these will be illustrated throughout.

\subsection{Results}

We start with a very simple result, that we report formally for the convenience of the reader.

\begin{Proposition}\label{prop: gen}Let $d$, $d'$ be positive integers, with $d>d'$. If $(\psi,\{b_{n,d} \})$ is a $d$-Schoenberg pair, then the $d'$-Schoenberg sequence of coefficients of $\psi$ is uniquely determined as follows.
\begin{enumerate}\itemsep=0pt
 \item[$(i)$] For $d' \ge 2$,
 \begin{gather}\label{eq: rel2}
 b_{n,d'}=\frac{\kappa(n,d')}{C_n^{(d-1)/2}(1)}\sum_{n=0}^\infty b_{n,d}\int_0^\pi C_n^{(d-1)/2}(\cos \theta)C_n^{(d'-1)/2}(\cos \theta){\rm d}\theta,
 \end{gather}
 with $\kappa(n,d)$ as defined in \eqref{eq: kappa}.
 \item[$(ii)$] For $d'=1$,
\begin{gather}
b_{0,1}=\frac{1}{\pi}\sum_{n=0}^\infty b_{n,d}\int_0^\pi c_n^{(d-1)/2}(\cos \theta){\rm d}\theta,\nonumber\\
b_{n,1}=\frac{2}{\pi}\sum_{n=0}^\infty b_{n,d} \int_0^\pi c_n^{(d-1)/2}(\cos \theta )\cos(n\theta){\rm d}\theta, \qquad n\geq 1.\label{eq: rel1}
\end{gather}
\end{enumerate}
\end{Proposition}
\begin{proof}The identity \eqref{eq: rel2} is obtained substituting \eqref{eq: spectral_rep} into~\eqref{eq: d-schoenberg}, whereas the identities~\eqref{eq: rel1} are obtained substituting~\eqref{eq: spectral_rep} into \eqref{eq: d-schoenberg1}. In both cases, exchanging integral and series is allowed owing to both bounded and uniform convergence of the series \eqref{eq: spectral_rep}.
\end{proof}

We are not aware of any closed-form expression for the integrals appearing in \eqref{eq: rel1} and \eqref{eq: rel2}, and therefore of the relationships between the sequences $\{b_{n,d}\}_{n=0}^\infty$ and $\{b_{n,d'}\}_{n=0}^\infty$ attached to a~$d'$-Schoenberg pair $(\psi,b_{n,d'})$, apart from the specific case where $d'=d+2$. Indeed, \cite{gneiting2} provides a closed-form expression for $\{b_{n,d+2}\}_{n=0}^\infty$ as a function of $\{b_{n,d}\}_{n=0}^\infty$ that is given by \eqref{eq: rel_gegenb1}--\eqref{eq: rel_gegenb3}. Our first main results provide an explicit expression for the inverse function

\begin{Theorem}\label{prop: spec1}If $(\psi, \{b_{n,3}\})$ is a $3$-Schoenberg pair, then the $1$-Schoenberg sequence of coefficients of $\psi$ is given by
\begin{gather}\label{eq: rel13a}
 b_{0,1}=\sum_{j=0}^\infty \frac{1}{2j+1}b_{2j,3},\\
 b_{n,1}=\sum_{j=0}^\infty \frac{2}{n+2j+1}b_{n+2j,3},\qquad n \geq 1. \label{eq: rel13b}
\end{gather}
\end{Theorem}

\begin{proof}From identity \eqref{eq: rel_gegenb2}, if $(\psi, \{b_{n,3}\})$ is a $3$-Schoenberg pair, we have that
\begin{gather*}
 \frac{2}{n+1}b_{n,3}=b_{n,1}-b_{n+2,1},\qquad n\geq 1.
\end{gather*}
Hence, for every nonnegative integer $j$, and for any positive integer $n$,
\begin{gather}\label{eq: pr1}
 \frac{2}{n+2j+1}b_{n+2j,3}=b_{n+2j,1}-b_{n+2j+2,1}.
\end{gather}
Summing up both sides of \eqref{eq: pr1} from $0$ to $m$, we obtain
\begin{gather}\label{eq: pr2}
 \sum_{j=0}^m \frac{2}{n+2j+1}b_{n+2j,3}= \sum_{j=0}^m \big ( b_{n+2j,1}-b_{n+2j+2,1} \big ) ,\qquad m\geq 1.
\end{gather}
We now use the fact that the right-hand side in equation~(\ref{eq: pr1}) is telescopic. Hence, \eqref{eq: pr2} can be written as
\begin{gather}\label{eq: pr3}
 \sum_{j=0}^m \frac{2}{n+2j+1}b_{n+2j,3}= b_{n,1}-b_{n+2m+2,1},\qquad m\geq 1.
\end{gather}
Since $\psi$ belongs to $\Psi_1$, the series $\sum\limits_{n=0}^\infty b_{n,1}$ converges to $1$ and, therefore, the sequence $\{b_{n,1}\}_{n=0}^\infty$ converges to zero. We can thus take the limit for $m \to \infty$ in equation~\eqref{eq: pr3} and this will provide~\eqref{eq: rel13b}. In particular, we now take $n=2$ to deduce that $b_{2,1}= 2\sum\limits_{j=1}^\infty b_{2j,3}/\{1+2j\}$ which combined with~\eqref{eq: rel_gegenb1} yields~\eqref{eq: rel13a}.
\end{proof}

We are now able to provide an extension of Theorem \ref{prop: spec1} for $d>3$. For a positive integer $m$ and $x>0$, $(x)^m$ will denote the standard rising factorial (Pochhammer symbol).

\begin{Theorem}\label{prop: spec2}Let $d \ge 2$ be a positive integer. If $(\psi, \{ b_{n,d+2} \})$ is a $(d+2)$-Schoenberg pair, then the $d$-Schoenberg sequence of coefficients $\{b_{n,d}\}_{n=0}^\infty$ of $\psi$ is given by
\begin{gather*}
 b_{n,d}=
 \sum_{j=0}^\infty w_{j,n,d}
 b_{n+2j,d+2}, \qquad n\geq 1,
\end{gather*}
where
\begin{align*}
 w_{0,n,d}&=\frac{d(2n+d-1)}{(n+d-1)(n+d)}, \\
 w_{j,n,d}&=d(2n+d-1) \frac{(n/2+1/2)^{(j)}(n/2+1)^{(j)}}{(n/2+(d-1)/2)^{(j+1)}(n/2+d/2)^{(j+1)}}, \qquad j\geq 1.
\end{align*}
\end{Theorem}

\begin{proof}We give a constructive proof. Define
\begin{gather*}
 a_{n,d} := \frac{d(2n+d-1)}{(n+d-1)(n+d)},\qquad u_{n,d} :=2n+d-1,\qquad v_{n,d} :=\frac{(n+1)(n+2)}{(n+d-1)(n+d)}.
\end{gather*}
We can now rewrite equation \eqref{eq: rel_gegenb3} as
\begin{gather}\label{eq: rel_g}
 a_{n,d}b_{n,d+2}=b_{n,d}-\frac{u_{n,d}}{u_{n+2,d}} v_{n,d}b_{n+2,d}, \qquad n\geq 0.
\end{gather}
Identity \eqref{eq: rel_g} shows that for every pair of nonnegative integers $(j,n)$, it is true that
\begin{gather}\label{eq: rel_g1}
 a_{n+2j,d}b_{n+2j,d+2}=b_{n+2j,d}-\frac{u_{n+2j,d}}{u_{n+2j+2,d}} v_{n+2j,d}b_{n+2j+2,d}.
\end{gather}
Multiplying each side of \eqref{eq: rel_g1} by $(u_{n,d}/u_{n+2j,d})\prod\limits_{l=0}^{j-1} v_{n+2l,d}$ and summing up both sides from~$0$ to~$m$, we obtain
\begin{gather*}
\sum_{j=0}^m\left(\prod_{l=0}^{j-1} v_{n+2l,d}\right)\frac{u_{n,d}}{u_{n+2j,d}} a_{n+2j,d}b_{n+2j,d+2}\\
\qquad{} =u_{n,d}\sum_{j=0}^m \left(\prod_{l=0}^{j-1} v_{n+2l,d}\right) \left( \frac{b_{n+2j,d}}{u_{n+2j,d}}- \frac{b_{n+2j+2,d}}{u_{n+2j+2,d}} v_{n+2j,d}\right).
\end{gather*}
Since the sum in the right-hand side is telescopic, we are left with
\begin{gather}\label{eq: rel_g2}
\sum_{j=0}^m \left(\prod_{l=0}^{j-1} v_{n+2l,d}\right)\frac{u_{n,d}}{u_{n+2j,d}} a_{n+2j,d}b_{n+2j,d+2}\nonumber\\
\qquad{} =b_{n,d}-\frac{u_{n,d}}{u_{n+2m+2,d}}\left(\prod_{l=0}^{m} v_{n+2l,d}\right) b_{n+2j+2,d} .
\end{gather}
At this stage, note that
\begin{gather}\label{eq: proof-step0}
 v_{n,d}-1= -(d-2)\frac{2n+d+1}{(n+d-1)(n+d)}\leq -(d-2)\frac{1}{n+d-1}, \qquad n\geq 0.
\end{gather}
We can now show that $\prod\limits_{l=0}^\infty v_{n+2l,d}\in \{0,1\}$.
Indeed, if $d=2$, then $v_{n,d}=1$ for each $n\geq 0$
and, therefore, $\prod\limits_{l=0}^\infty v_{n+2l,2}=1$. If $d>2$, then by \eqref{eq: proof-step0}
\begin{gather*}
\prod_{l=0}^{m} v_{n+2l,d}= \exp\left[\sum_{l=0}^m \log(v_{n+2l,d}) \right]\\
\hphantom{\prod_{l=0}^{m} v_{n+2l,d}}{} \leq \exp\left[\sum_{l=0}^m (v_{n+2l,d}-1) \right] \leq \exp\left[-(d-2) \sum_{l=0}^m \frac{1}{n+2l+d-1} \right],
\end{gather*}
and, therefore, $\prod\limits_{l=0}^{\infty} v_{n+2l,d}=0$. Since $\psi\in \Psi_d$, the sequence $\{b_{n,d}\}_{n=0}^\infty$ converges to zero, while
\begin{gather*}
\lim_{m\to \infty}\frac{u_{n,d}}{u_{n+2m+2,d}}=0,\qquad n\geq 0.
\end{gather*}
So, letting $m \to \infty$ in \eqref{eq: rel_g2} yields
\begin{gather*}
 b_{n,d}=\sum_{j=0}^\infty \left(\prod_{l=0}^{j-1} v_{n+2l,d}\right)\frac{u_{n,d}}{u_{n+2j,d}} a_{n+2j,d}b_{n+2j,d+2}, \qquad n\geq 0.
\end{gather*}
Finally, direct computation shows that for $n\geq 0$ and $j\geq 1$,
\begin{gather*}
 \prod_{l=0}^{j-1} v_{n+2l,d}=\frac{(n/2+1/2)^{(j)}(n/2+1)^{(j)}}{(n/2+(d-1)/2)^{(j)}(n/2+d/2)^{(j)}},
\end{gather*}
and
\begin{gather*}
 \frac{u_{n,d}}{u_{n+2j,d}} a_{n+2j,d}=\frac{d(2n+d-1)}{(n+2j+d-1)(n+2j+d)}.
\end{gather*}
The proof is completed.
\end{proof}

\section{Schoenberg sequences on complex spheres} \label{3}

In analogy with the results obtained in Section~\ref{2}, we consider similar results on complex spheres.

\subsection{Background and notation}
For a positive integer $q$, denote by $\Omega_{2q}$ the unit sphere in $\mathbb{C}^q$. A mapping $C\colon \Omega_{2d}\times \Omega_{2q}\to \mathbb{C}$ is {\em positive definite} if
\begin{gather*}
\sum_{i,j=1}^n c_i \bar{c}_j C(\bm{z}_i,\bm{z}_j) \geq 0.
\end{gather*}
for all $n\geq 1$, distinct points $\bm{z}_1,\dots,\bm{z}_n$ of $\Omega_{2q}$ and complex numbers $c_1,\hdots,c_n$.
Let ``$\cdot$'' denote the usual inner product in $\mathbb{C}^q$. If $q\geq 2$ and $B[0,1]=\{z\in \mathbb{C}\colon z \cdot z \leq 1 \}$, the function $C$ is called {\em isotropic} if
\begin{gather}\label{eq: bizonal}
C(\bm{z}_1,\bm{z}_2)=\varphi(\bm{z}_1\cdot \bm{z}_2),\qquad \bm{z}_1,\bm{z}_2 \in \Omega_{2q},
\end{gather}
for some function $\varphi\colon B[0,1]\to \mathbb{C}$. This nomenclature is not universal but it is quite adequate in our setting. Observe that in the case $q=1$, if $z,w \in \Omega_2$, then $z \cdot z \in \Omega_2$. Hence, the previous definition becomes an extreme case once the domain of $\varphi$ needs to be $\Omega_2$ itself.

Keeping the analogy with the previous section, for $q\geq 2$, we call $\Upsilon_{2q}$ the class of continuous functions $\varphi$, with $\varphi(1)=1$ such that $C$ in~(\ref{eq: bizonal}) is positive definite.
We also denote by $\Upsilon_{2q}^{+}$ the class of functions $\varphi$ belonging to $\Upsilon_{2q}$ such that $C$ in~(\ref{eq: bizonal}) is strictly positive definite. Both classes $\Upsilon_{2q}$ and $\Upsilon_{2q}^+$ are nested, that is, if $q \leq q'$, then $\Upsilon_{2q'} \subset \Upsilon_{2q}$ and $\Upsilon_{2q'}^+\subset \Upsilon_{2q}^+$.

To present the characterization of the class $\Upsilon_{2q}$ described in \cite{MenegattoPeron01}, we denote by $R^{q-2}_{m,n}$ the disk polynomial of bi-degree $(m,n)$ with respect to the nonnegative integer $q-2$.
The set $\{R^{q-2}_{m,n}\colon m,n=0,1,\ldots\}$ is a complete orthogonal system in $L^2(B[0,1],\nu_{q-2})$, with
\begin{gather}\label{eq: measure}
{\rm d} \nu_{q-2} (z)=\frac{q-1}{\pi} \big(1-|z|^2\big)^{q-2} {\rm d} x {\rm d} y, \qquad z=x+ \mathsf{i} y\in B[0,1].
\end{gather}
In particular,
\begin{gather} \label{neymar}
\int_{B[0,1]} R^{q-2}_{m,n}(z) \overline{R^{q-2}_{k,l}(z)} {\rm d} \nu_{q-2}(z)= \frac{\delta_{mk}\delta_{nl}}{h^{q-2}_{m,n}},
\end{gather}
where
\begin{gather}\label{eq: h}
h^{q-2}_{m,n}=\frac{m+n+q-1}{q-1} \binom{m+q-2}{q-2} \binom{n+q-2}{q-2}.
\end{gather}
Expressions and main properties of disk polynomials can be found in~\cite{Wuensche05} and in references quoted there. We recall the following recursion satisfied for every $z$ in $B[0,1]$, $m\geq 1$ and $n\geq 0$~\cite{Menegatto14}:
\begin{gather}\label{eq: disk recursion}
\big(1-|z|^2\big) R^{q-1}_{m-1,n}(z)= \frac{q-1}{m+n+q-1}\big[ R^{q-2}_{m-1,n}(z)-R^{q-2}_{m,n+1}(z)\big].
\end{gather}
For every continuous function $\varphi\colon B[0,1]\to \mathbb{C}$ and every triplet $(m,n,q)$ of nonnegative integers, we can define
\begin{gather}\label{eq: coeff}
a^{q-2}_{m,n}:=h^{q-2}_{m,n} \int_{B[0,1]} \varphi (z) \overline{R^{q-2}_{m,n}(z)} {\rm d}\nu_{q-2}(z).
\end{gather}
The functions belonging to the class $\Upsilon_{2q}$ are uniquely characterized through the expansion~\cite{MenegattoPeron01}
\begin{gather} \label{jennifer}
\varphi(z)= \sum_{m,n=0}^\infty a^{q-2}_{m,n} R^{q-2}_{m,n}(z), \qquad z \in B[0,1],
\end{gather}
where $a^{q-2}_{m,n}\geq 0$, $m,n\in \mathbb{Z}_+$ and $\sum\limits_{m,n=0}^\infty a^{q-2}_{m,n}=1$. Following Section~\ref{2}, we finally define a~{\em $2q$-Schoenberg pair} $\big(\varphi, \big\{a_{m,n}^{q-2} \big\}\big)$ any function belonging to the class $\Upsilon_{2q}$ with expansion defined according to (\ref{jennifer}). In this case, the double sequence $\big\{a_{m,n}^{q-2}\big\}_{m,n=0}^\infty$ will be called the {\em $2q$-Schoenberg sequence of coefficients }of $\varphi$.

\subsection{Results}

Since the classes $\Upsilon_{2q}$ are nested, here we prove a recursive relation among the coefficients $a_{m,n}^{q-1}$ and $a_{m,n}^{q-2}$ attached to a $2(q+1)$ Schoenberg pair $\big(\varphi, \big\{a_{m,n}^{q-1}\big\}\big)$ that resembles~\eqref{eq: rel_gegenb3}.

\begin{Proposition} If $\big(\varphi, \big\{a_{m,n}^{q-1}\big\}\big)$ is a $2(q+1)$-Schoenberg pair, then for $m-1,n \geq 0$,
 \begin{gather}\label{eq: rel_disk}
 a^{q-1}_{m-1,n}= \frac{(m+q-2)(n+q-1)}{(q-1)(m+n+q-2)} a^{q-2}_{m-1,n} -\frac{m(n+1)}{(q-1)(m+n+q)} a^{q-2}_{m,n+1}.
 \end{gather}
\end{Proposition}

\begin{proof}Equation \eqref{eq: h} shows that
\begin{gather}\label{eq: h2}
 h^{q-2}_{m-1,n} =\frac{(m+n+q-2)m(n+1)}{(m+n+q)(m+q-2)(n+q-1)} h^{q-2}_{m,n+1}, \qquad m-1,n\geq 0,\\
 h^{q-2}_{m-1,n} =\frac{(m+n+q-2)q(q-1)^2}{(m+n+q-1)q(m+q-2)(n+q-1)} h^{q-1}_{m-1,n}, \qquad m-1,n\geq 0. \label{eq: h3}
\end{gather}
We now multiply both sides of \eqref{eq: disk recursion} by $h^{q-2}_{m-1,n} \varphi(z)$ and integrate with respect to the measure~$\nu_\alpha$ defined in~(\ref{neymar}). After we use~\eqref{eq: coeff} and \eqref{eq: h2}, we obtain
\begin{gather}
 h^{q-2}_{m-1,n} \int_{B[0,1]} \big(1-|z|^2\big) R^{q-1}_{m-1,n}(z) f\varphi (z) {\rm d} \nu_{q-2}(z)\nonumber\\
\qquad{} = \frac{q-1}{m+n+q-1} \left[a^{q-2}_{m-1,n}- \frac{(m+n+q-2)m(n+1)}{(m+n+q)(m+q-2)(n+q-1)} a^{q-2}_{m,n+1}\right].\label{eq: step0}
\end{gather}
However, equation \eqref{eq: measure} yields the equality
\begin{gather*}
\big(1-|z|^2\big)d \nu_{q-2}= \frac{q-1}{q}d \nu_{q-1}.
\end{gather*}
Therefore, by \eqref{eq: h3} and \eqref{eq: coeff}, the left-hand side of \eqref{eq: step0} is equal to
\begin{gather*}
\frac{(m+n+q-2)(q-1)^2}{(m+n+q-1)(m+q-2)(n+q-1)}
 a^{q-1}_{m-1,n},
\end{gather*}
so that \eqref{eq: step0} becomes
\begin{gather*}
 \frac{(m+n+q-2)(q-1)}{(m+q-2)(n+q-1)}
 a^{q-1}_{m-1,n}=
 a^{q-2}_{m-1,n}-
 \frac{(m+n+q-2)m(n+1)}{(m+n+q)(m+q-2)(n+q-1)}
 a^{q-2}_{m,n+1}.
\end{gather*}
This yields \eqref{eq: rel_disk}.
\end{proof}

Here is the main result of the section.

\begin{Theorem} \label{maincomplex}If $\big(\varphi,\big\{a_{m,n}^{q-1}\big\}\big)$ is a $2(q+1)$-Schoenberg pair, then the $2q$-Schoenberg sequence of coefficients $\big\{a^{q-2}_{m,n}\big\}_{m,n=0}^\infty$ of $\varphi$ is given by
\begin{gather*}
 a^{q-2}_{m,n} = \sum_{j=0}^\infty v^{q-2}_{j,m+1,n} a^{q-1}_{m+j,n+j},\qquad m,n\geq 0,
\end{gather*}
where
\begin{gather*}
 v^{q-2}_{j,m,n}:=\frac{m^{(j)}(n+1)^{(j)}(m+n+q-2)}{(m+q-2)^{(j)}(n+q-1)^{(j)}(m+n+2j+q-2)}, \qquad j\geq 0.
\end{gather*}
\end{Theorem}
\begin{proof}First of all we introduce the following notations:
\begin{gather}
 u^{q-2}_{m,n} := \frac{(q-1)(m+n+q-2)}{(m+q-2)(n +q-1)},\qquad m,n\geq 0,\nonumber\\
 w^{q-2}_{m,n}:= \frac{(m+n+q-2)m(n+1)}{(m+q-2)(n+q-1)(m+n+q)},\qquad m,n\geq 0.\label{eq: uw}
\end{gather}
In this way, \eqref{eq: rel_disk} becomes
\begin{gather}\label{eq: rel_disk2}
u^{q-2}_{m,n} a^{q-1}_{m-1,n}= a^{q-2}_{m-1,n}-w^{q-2}_{m,n} a^{q-2}_{m,n+1}, \qquad m-1,n\geq 0.
\end{gather}
By \eqref{eq: rel_disk2}, we have that for every triplet $(j,m,n)$ of nonnegative integers,
\begin{gather}\label{eq: rel_disk3}
 u^{q-2}_{m+j,n+j} a^{q-1}_{m+j-1,n+j}= a^{q-2}_{m+j-1,n+j}-w^{q-2}_{m+j,n+j} a^{q-2}_{m+j,n+j+1}.
\end{gather}
Now, we can multiply each side of \eqref{eq: rel_disk3} by the product $\prod\limits_{l=1}^j w^{q-2}_{m+l-1,n+l-1}$ and sum up each side from $0$ to $k$,
obtaining that
\begin{gather*}
\begin{split}&  \sum_{j=0}^k \left (\prod_{l=1}^j w^{q-2}_{m+l-1,n+l-1} \right) u^{q-2}_{m+j,n+j} a^{q-1}_{m+j-1,n+j}\\
& \qquad{} = \sum_{j=0}^k \left (\prod_{l=1}^j w^{q-2}_{m+l-1,n+l-1} \right) \left(a^{q-2}_{m+j-1,n+j}-w^{q-2}_{m+j,n+j} a^{q-2}_{m+j,n+j+1} \right).
\end{split}
\end{gather*}
 Since the sum in the right-hand side is telescopic, we are reduced to
\begin{gather}
 \sum_{j=0}^k \left (\prod_{l=1}^j w^{q-2}_{m+l-1,n+l-1} \right) u^{q-2}_{m+j,n+j} a^{q-1}_{m+j-1,n+j}\nonumber\\
 \qquad{} = a^{q-2}_{m-1,n} - \left (\prod_{j=0}^k w^{q-2}_{m+j,n+j} \right) a^{q-2}_{m+k,n+k+1}.\label{eq: proofff}
\end{gather}
 Since
\begin{gather*}
 \prod_{j=0}^k w^{q-2}_{m+j,n+j}\\
 \qquad{} \leq \exp\left[ -\sum_{j=0}^k \left( \frac{q-2}{m+j+q-2} + \frac{q-2}{n+j+q-1} + \frac{2}{m+n+2j+q} \right) \right], \qquad k\geq 0,
\end{gather*}
we end up with
\begin{gather*}
 \prod_{j=0}^\infty w^{q-2}_{m+j,n+j} =0.
\end{gather*}
Moreover, since
\begin{gather*}
 \sum_{k=0}^\infty a^{q-2}_{l+k,j+k+1}\leq \sum_{m,n=0}^\infty a^{q-2}_{m,n}<\infty, \qquad j,l\geq 0,
\end{gather*}
we have that $\lim\limits_{k\to \infty} a^{q-2}_{l+k,j+k+1}=0$, for $j,l\geq 0$. Therefore, letting $k \to \infty$, \eqref{eq: proofff} leads to
\begin{gather*}
 a^{q-2}_{m-1,n}=\sum_{j=0}^\infty
 \left (\prod_{l=1}^j w^{q-2}_{m+l-1,n+l-1} \right)
 u^{q-2}_{m+j,n+j} a^{q-1}_{m+j-1,n+j},\qquad m-1,n\geq 0,
\end{gather*}
which in turn by \eqref{eq: uw} yields the desired result.
\end{proof}

\section[Applications involving the classes $\Psi_d^+$ and $\Upsilon_{2q}^{+}$]{Applications involving the classes $\boldsymbol{\Psi_d^+}$ and $\boldsymbol{\Upsilon_{2q}^{+}}$} \label{4}

In this section, we present applications of the previous results involving the classes $\Psi_d^+$ and $\Upsilon_{2q}^{+}$.

\begin{Theorem} \label{inclu}Let $q,q'\ge2$ be integers. The following assertions hold:
\begin{enumerate}\itemsep=0pt
\item[$(i)$] If a function $\varphi$ belongs to $\Upsilon_{2q}^+\cap \Upsilon_{2q'}$, then $\varphi$ belongs to $\Upsilon_{2q'}^+$.
\item[$(ii)$] If a function $\varphi$ belongs to $(\Upsilon_{2q}\setminus\Upsilon_{2q}^+)\cap \Upsilon_{2q'}$, then $\varphi$ belongs to $\Upsilon_{2q'} \setminus \Upsilon_{2q'}^+$.
\end{enumerate}
\end{Theorem}

\begin{proof} $(i)$ If $q \geq q'$, the assertion follows from the inclusion $\Upsilon_{2q}^+\subset \Upsilon_{2q'}^+$. So, we may assume that $q<q'$. If $\varphi \in \Upsilon_{2q}^+$, Theorem~1.1 in~\cite{guella2018} reveals that the $2q$-Schoenberg sequence of coefficients $\big\{a_{m,n}^{q-2}\big\}_{m,n=0}^\infty$ of $\varphi$ has the following property: $\big\{m-n\colon a_{m,n}^{q-2}>0\big\}$ intersects every arithmetic progression of $\mathbb{Z}$. Taking into account that $\varphi \in \Upsilon_{2(q+1)}$ and the fact that $ v^{q'-2}_{j,m+1,n}>0$ for all~$j$, Theorem~\ref{maincomplex} shows that $a_{m,n}^{q-2} >0$ if and only if $a_{m+j,n+j}^{q-1}>0$, for at least one $j\geq 0$. In particular, the set $\big\{m-n\colon a_{m,n}^{q-1}>0\big\}$ intersects every arithmetic progression of $\mathbb{Z}$ as well. In other words, $\varphi \in \Upsilon_{2(q+1)}^+$, due to Theorem 1.1 in \cite{guella2018} once again. If $q+1=q'$, $\varphi \in \Upsilon_{2q'}^+$ and we are done. Otherwise, we iterate this procedure until we reach the desired conclusion.

$(ii)$ Assume $\varphi \in (\Upsilon_{2q}\setminus\Upsilon_{2q}^+)\cap \Upsilon_{2q'}$. If $\varphi \in \Upsilon_{2q'}^+$, then $\varphi \in \Upsilon_{2q'}^+\cap \Upsilon_{2q}$ and $(i)$ would imply that $\varphi \in \Upsilon_{2q}^+$, a contradiction.
\end{proof}

A similar result holds for real spheres with a similar proof. In particular, if $\psi$ belongs to $\Psi_d \cap \Psi_{d'}^+$, then $\psi$ belongs to $\Psi_d^+$. However, this result was proved earlier in \cite[Corollary~1]{gneiting2} via a slightly different argument.

Theorem \ref{inclu} allows the following obvious consequences. If $\varphi$ is a function in $\Upsilon_{2q}$, we write~$\varphi_r$ to indicate the restriction of $\varphi$ to $[-1,1]$.

\begin{Corollary} For $d\geq 1$ and $q\geq 2$, the following assertions hold:
\begin{enumerate}\itemsep=0pt
\item[$(i)$] If a function $\varphi$ belongs to $\Upsilon_{2q}$ and $\varphi_r \circ \cos$ belongs to $\Psi_d^+$, then $\varphi_r \circ \cos$ belongs to $\Psi_{2q-1}^+$.
\item[$(ii)$] If a function $\varphi$ belongs to $\Upsilon_{2q}^+$ and $\varphi_r \circ \cos$ belongs to $\Psi_d$, then $\varphi_r \circ \cos$ belongs to $\Psi_{d}^+$.
\end{enumerate}
\end{Corollary}

\begin{proof} It suffices to observe that if $\varphi \in \Upsilon_{2q}$ (respectively, $\Upsilon_{2q}^+$), then $f_r \circ \cos \in \Psi_{2q-1}$ (respectively, $\Psi_{2q-1}^+$) and to apply the remark in the paragraph preceding the theorem.
\end{proof}

\section{Discussion}

This paper contributes to the literature about the classes $\Psi_d$, $\Upsilon_{2q}$ and $\Upsilon_{2q}^+$ in terms of their Schoenberg sequences. Yet, there are many challenges that involve Schoenberg sequences, for instance in product spaces. Berg and Porcu~\cite{berg-porcu} consider the analogue of Schoenberg pairs introduced in this paper, but on the product space $\mathbb{S}^d \times G$, for $G$ a locally compact group. Generalizations of the results in \cite{berg-porcu} have been provided by \cite{guella1}. It would be very interesting to inspect whether the results provided in this paper can be generalized to these cases. Another important challenge would be to inspect the Schoenberg pairs related to matrix-valued kernels (see open problem 2 in \cite{PAF18}).

\subsection*{Acknowledgements}
The authors are grateful to the Associate Editor and to the referees for their comments that allowed for an improved version of the manuscript. Research of Valdir A.~Menegatto was partially supported by FAPESP, grant 2016/09906-0. Emilio Porcu is partially supported by grant FONDECYT 1130647 from Chilean Ministry of Education, and by Millennium Science Initiative of the Ministry of Economy, Development, and Tourism, grant ``Millenium Nucleus Center for the Discovery of Structures in Complex Data''.

\pdfbookmark[1]{References}{ref}
\LastPageEnding

\end{document}